\documentclass[11pt,a4paper]{article}

\usepackage[utf8]{inputenc}
\usepackage[T1]{fontenc}
\usepackage[english]{babel}
\usepackage{amsmath, amssymb, amsthm}
\usepackage{hyperref}
\usepackage{geometry}
\usepackage{xcolor}
\usepackage{enumitem}
\newtheorem{theorem}{Theorem}[section]

\newtheorem{proposition}[theorem]{Proposition}

\theoremstyle{definition}
\newtheorem{definition}[theorem]{Definition}
\newtheorem{example}[theorem]{Example}

\theoremstyle{remark}
\newtheorem{remark}[theorem]{Remark}

\newtheorem{theoremx}{\textbf{Theorem}}

\title{Classification of normal toric surfaces resolved by a single Nash blowup}
\author{Amador Cruz-Fuentes\thanks{Research supported by SECIHTI project CF-2023-G-33.}}
\date{\today}

\begin{document}

\maketitle

\begin{abstract}
We present a complete classification of normal toric surfaces that are resolved by a single normalized Nash blowup. Likewise, we obtain a complete classification of those resolved by a single Nash blowup. In both cases, the classification is expressed in terms of the continued fraction associated with the normal toric surface.
\end{abstract}

\section*{Introduction}

The Nash blowup of an algebraic variety is a modification that replaces its singular points by limits of tangent spaces of nearby smooth points.
It was proposed that singularities could be resolved by iterating this process or by composing it with the normalization \cite{Semple, Nobile, mark, GonzalezSprinberg1977ventailsED}.

The starting point for the study of the resolution problem via Nash blowups was the work of A. Nobile \cite{Nobile}, who proved that the Nash blowup coincides with the blowup of the variety along an explicitly described ideal. In the same work, Nobile also showed that, over fields of characteristic zero, the Nash blowup is an isomorphism if and only if the variety is smooth. Later on, an analogous result was proved over postiive characteristic fields and normal varieties by D. Duarte and L. Núñez-Betancourt \cite{DN}. These results confirmed the relevance of studying Nash blowups as a tool for resolution.

In both the Nash blowup and its normalized version, several authors have provided evidence suggesting that singularities might indeed be resolved by iterating these processes \cite{Nobile, rebassoo, GonzalezSprinberg1977ventailsED, Gonzaleznashnorma, hironakanash, mark, sprinfrac, Gize, grig, duarte1, ata, Duante-Green, DJN, thais, DuarteSnoussi2025, CDLLfree}.
Nevertheless, F. Castillo, D. Duarte, M. Leyton, and A. Liendo recently presented explicit examples in dimension four and higher of normal affine varieties whose Nash blowup and normalized Nash blowup contain an affine chart isomorphic to the original variety \cite{CDLAL, CDLLcomp}.
This shows that, in general, neither of the two procedures resolves singularities. Moreover, a counterexample to the question on non-normalized Nash blowups in dimension three was found recently \cite{castillo2025nonnormalizednashblowupfails, CDLLcomp}. Considering the known counterexamples, to fully answer Nash’s questions, attention is currently focused on non-normalized Nash blowups in dimension two and on normalized Nash blowups in dimension three. Since these counterexamples are toric, it is natural to begin by analyzing the toric case.

In particular, it remains unknown whether iterating the Nash blowup resolves singularities in dimension two.
In this work, we focus on characterizing the hypothetical final step of the process in the toric case, that is, describing conditions on a toric surface under which a single Nash blowup suffices to obtain a resolution. We give an answer to this question in the case of normal toric surfaces.\looseness=-1 

Our first main theorem concerns the normalized Nash blowup.

\begin{theoremx}
    Let $\operatorname{char}(\mathbb{K}) = p \ge 0$.  
    The normalized Nash blowup of a normal toric surface is smooth if and only if its associated continued fraction
    $
        [1, a_2, \dots, a_r]
    $
    is one of the following:

    \begin{enumerate}[label=(\arabic*)]
        \item $[1,2]$ if $r = 2$.
        \item $[1,2,2]$ if $r = 3$.
        \item $[1,2,2,2]$, $[1,2,3,2]$ or $[1,2,4,2]$ if $r = 4$.
        \item $[1,2,a_3,\dots,a_{r-1},2]$ for $r \ge 5$, where $(a_i, a_{i+1}) \in \{(2,2),(2,3),(2,4),(3,2),(4,2)\}$ for all $3 \le i \le r-2$.
    \end{enumerate}
\end{theoremx}

The second main theorem deals with the non-normalized Nash blowup.

\begin{theoremx}
    Assume $\operatorname{char}(\mathbb{K})=0$.  
The Nash blowup of a normal toric surface is smooth if and only if its associated continued fraction 
\(
[1,a_2,\dots,a_r]
\) 
is one of the following:

\begin{enumerate}[label=(\arabic*)]
    \item $[1,2]$ if $r=2$.
    \item $[1,2,2]$ if $r=3$.
    \item $[1,2,2,2]$ or  $[1,2,3,2]$, if $r=4$.
    \item $[1,2,a_3,\dots, a_{r-1},2]$ for $r \geq 5$, where $(a_i,a_{i+1}) \in \{(2,2),(2,3),(3,2)\}$ for all $3 \leq i \leq r-2$.
\end{enumerate}
\end{theoremx}

As these results illustrate, we follow the classical approach based on continued fractions, which has played a central role in the study of singularities of normal toric surfaces \cite{Oda1988, Fulton, CLS, ata}. The approach adopted in this article follows the treatment given in \cite{ata}.\\

\medskip
\noindent\textbf{Acknowledgments.} The author would like to express his gratitude to Daniel Duarte for his careful reading of the manuscript and for his valuable comments, suggestions, and advice. He is also sincerely thankful for the time and guidance provided throughout the preparation of this work. 

\section{Nash blowups and toric varieties}
We begin by recalling the definition of the Nash blowup of an algebraic variety, followed by the definition of an affine toric variety.

\begin{definition}
Let $\mathbb{K}$ be an algebraically closed field of arbitrary characteristic. Let $X \subset \mathbb{K}^n$ be an irreducible algebraic variety of dimension $d$. Consider the map
    $$
    G : X \setminus \operatorname{Sing}(X) \longrightarrow X \times Gr(d,n),\quad
    G(p) = (p, T_p(X)),
    $$
    where $T_p(X)$ is the tangent space of $X$ at $p$. Denote by $X^{\ast}$ the Zariski closure of the image of $G$ in $X \times Gr(d,n)$, and let $\nu$ be the restriction to $X^{\ast}$ of the projection $X \times Gr(d,n) \to X$. The Nash blowup of $X$ is the pair $(X^{\ast}, \nu)$. 
\end{definition}

In this paper, we are interested in the Nash blowup of toric varieties.

\begin{definition}
Let $\Gamma \subset \mathbb{Z}^d$ be the semigroup generated by $\{\gamma_1,\dots , \gamma_n\}$. Consider the $\mathbb{K}$-algebra homomorphism $\pi _{\Gamma}:\mathbb{K}[x_1,\dots, x_n]\to \mathbb{K}[t_1^{\pm 1},\dots, t_d^{\pm 1}]$, defined by $x_i \mapsto t^{\gamma_i}$. The variety $X_{\Gamma}=V(\operatorname{Ker}(\pi_{\Gamma}))$ is called the toric variety defined by $\Gamma$, and its coordinate ring is denoted $\mathbb{K}[x^{\Gamma}]$.
\end{definition}

Given a finitely generated semigroup $\Gamma \subset \mathbb{Z}^d$, its saturation is given by $\check{\sigma}\cap \mathbb{Z}^d$, where $\check{\sigma}$ is the cone generated by $\Gamma$. The corresponding toric variety is denoted $X_{\sigma}$, which is the normalization of $X_{\Gamma}$.

In the context of toric varieties, the Nash blowup can be described combinatorially. In characteristic zero, this description relies on the logarithmic Jacobian ideal introduced by G. González-Sprinberg in \cite{GonzalezSprinberg1977ventailsED}. Later, in \cite{DJN}, this concept was generalized to prime characteristic by D. Duarte, J. Jeffries, and L. Núñez-Betancourt, allowing an extension of the description. 

\begin{definition}
    [\cite{GonzalezSprinberg1977ventailsED,GT,DJN}]
    Suppose $\operatorname{char}(\mathbb{K})=p\geq0$. Let $\Gamma \subset \mathbb{Z}^d$ be a semigroup generated by $ \{  \gamma_1,\dots, \gamma_n \}$ such that $\mathbb{Z}\Gamma=\mathbb{Z}^d$. The logarithmic Jacobian ideal modulo $p$ of $X_{\Gamma}$ is defined by:
   $$\mathcal{J}_p= \langle t^{\gamma_{i_1}+\cdots +\gamma_{i_d}}\mid \det(\gamma_{i_1}\cdots \gamma_{i_d})\not \equiv 0 \pmod{p}, \; 1\leq i_1<\cdots <i_d\leq n \rangle \subseteq \mathbb{K}[x^{\Gamma}].$$
\end{definition}

\begin{theorem}[\cite{GonzalezSprinberg1977ventailsED,GT,DJN}]\label{Nashtoric}
    Suppose $\operatorname{char}(\mathbb{K})=p\geq 0$. The Nash blowup of $X_{\Gamma}$ is isomorphic to the blowup of the variety along its logarithmic Jacobian ideal modulo $p$.
\end{theorem}
This theorem, together with results presented by P. González and B. Teissier in \cite{GT}, provides a combinatorial description of the Nash blowup of a toric variety. The description is presented below.

\section{Continued fractions in the study of toric surfaces}\label{descri}

It is well known that the minimal generating set of the semigroup associated with a normal toric surface can be described in terms of a continued fraction. This relation has been studied by several authors \cite{Oda1988, Fulton, CLS,ata}. In this section, we present this description and state some basic properties.

Given a strongly convex rational polyhedral cone  
\(\check{\sigma} \subset \mathbb{R}^2\), there exists an element of \(SL(2,\mathbb{Z})\) that transforms it into a cone of the form  
$$
 \mathbb{R}_{\geq 0}\{ (1,0), (P,Q)\},
$$  
where $0 \leq P < Q$ and $P, Q$ are coprime \cite[Proposition 4.9]{ata}. We call it the normal form of the cone.  
Without loss of generality, we assume that our initial cone is already in this form.  

Consider the rational number $P/Q$ and its continued fraction expansion:  
$$
\dfrac{P}{Q}=a_1-\dfrac{1}{a_2-\dfrac{1}{a_3-\dfrac{1}{\ddots -\dfrac{1}{a_r}}}}.
$$  

We denote this continued fraction as $[a_1,\dots , a_r]$ and call it the continued fraction associated with the normal toric surface $X_{\sigma}$.  

From this expression, we define the sequences:  
\begin{equation}\label{suce}
\begin{array}{ccc}
p_{-1}=0, & p_0=1, & p_i = a_i p_{i-1} - p_{i-2}, \quad 1 \leq i \leq r,\\[3pt]
q_0=0, & q_1=1, & q_i = a_i q_{i-1} - q_{i-2}, \quad 2 \leq i \leq r.
\end{array}
\end{equation}

The set $v_i=(p_i,q_i)$ for $0\leq i \leq r$ forms the Hilbert basis of the semigroup \(\check{\sigma} \cap \mathbb{Z}^2\), that is, the set of its irreducible elements \cite[Theorem 4.10]{ata}.

\begin{remark}\label{obs1}
For the continued fraction $[a_1,\dots , a_r]$ associated with a normal toric surface, we have:
\begin{itemize}
    \item $a_1=1$.
    \item $a_i > 1$ for each $i>1$.
\end{itemize}
\end{remark}

\begin{proposition}\cite[Section  4.1]{ata}\label{fc}
Let $[a_1,\dots , a_r]$ be a continued fraction.

\begin{enumerate}[label=(\arabic*)]
    \item For each $i\in \{1,\dots , r\}$, we have $[a_1,\dots , a_i]=\dfrac{p_i}{q_i}$.
    \item For $i>0$, we have $p_{i-1}q_i-p_iq_{i-1}=1$.
    \item For $i<j$, the denominator of the fraction $[a_{i+1},\dots , a_j]$ in lowest terms is $p_iq_j-p_jq_i$.
\end{enumerate}
\end{proposition}

\subsection{Combinatorial description of the Nash blowup of a toric surface}
From Theorem \ref{Nashtoric}, it is possible to give a combinatorial description of the Nash blowup of a toric variety, following the general approach proposed by P. González and B. Teissier for the blowup of a toric variety along any monomial ideal \cite{GT} . In this paper, we focus on the case of normal toric surfaces.\\

Consider the cone $\check{\sigma}=\mathbb{R}_{\geq 0}\{(1,0),(P,Q)\}$ in its normal  form, and let $\Gamma=\check{\sigma}\cap \mathbb{Z}^2$,  whose Hilbert basis is $\{v_0, v_1,\dots , v_r\}$.\\

The Newton polyhedron associated with $\check{\sigma}$ is 
$$
N_{\sigma}(\Gamma _p)= \operatorname{Conv}\{(v_j+v_k) +\check{\sigma}  \mid \det(v_j \ v_k)\not \equiv 0 \mod{p} \},
$$
where $p=\operatorname{char}(\mathbb{K})\geq0$ and $\operatorname{Conv}$ denotes the convex hull. The vertices of the Newton polyhedron are contained in $\{v_{i-1}+v_i \mid i\in \{1,\dots , r\}\}$, as shown in \cite[Proposition 4.12]{ata} and \cite[Proposition 2.4]{DJN}.\\

The combinatorial description of the Nash blowup and of the normalized Nash blowup of a toric surface (and more generally, of a toric variety) relies fundamentally on the Newton polyhedron and its vertices. According to the order in which the results are presented in this paper, we begin by describing the normalized Nash blowup and subsequently the Nash blowup.\\

\noindent\textbf{Normalized Nash blowup:}

Given a vertex $v_{i-1}+v_i$ of the Newton polyhedron, the localization of the polyhedron at $v_{i-1}+v_i$ is defined as the cone
\begin{eqnarray}\label{conosloc}
    \check{\sigma}(v_{i-1},v_i) = \mathbb{R}_{\geq 0}\bigl(N_{\sigma}(\Gamma_p) - (v_{i-1}+v_i)\bigr).
\end{eqnarray}

In this way, the set
$$
\bigl\{\,X_{\sigma(v_{i-1},v_i)} \;\bigm|\; v_{i-1}+v_i 
\text{ is a vertex of } N_{\sigma}(\Gamma_p)\,\bigr\}
$$  
constitute the affine charts of the normalized Nash blowup \cite{ata, GT, Gonzaleznashnorma}. In the following subsection, these localization cones are described in a more precise manner (see Subsection~\ref{locali}).\\

\noindent \textbf{Nash blowup:}

Given a vertex $v_{i-1}+v_i$ of $N_{\sigma}(\Gamma_0)$, consider  
\begin{eqnarray*}
    A(v_{i-1})=\{v_j-v_{i-1} \mid j\in \{0,\dots, r\}\setminus\{i-1, i\},\;\det(v_j\; v_i)\not = 0\},\\
    A(v_{i})=\{v_k-v_{i} \mid k\in \{0,\dots, r\}\setminus\{i-1, i\},\;\det(v_k\; v_{i-1})\not = 0\}.
\end{eqnarray*}
Let $\Gamma(v_{i-1},v_i)$ be the semigroup generated by $\{v_{i-1},v_i\}\cup  A(v_{i-1})\cup A(v_i)$. Then 
$$\{X_{\Gamma(v_{i-1},v_i) }\mid v_{i-1}+v_i\text{ is a vertex of } N_{\sigma}(\Gamma_0)\}$$
are the affine charts of the Nash blowup of $X_{\sigma}$ \cite{Spivakovsky2020, grig}. \\
\begin{remark}
 The toric surface $X_{\sigma(v_{i-1},v_i)}$ is the normalization of $X_{\Gamma(v_{i-1},v_i)}$. The Newton polyhedron $N_{\sigma}(\Gamma_p)$ does not depend on the characteristic of the field, in other words, $N_{\sigma}(\Gamma_p)=N_{\sigma}(\Gamma_0)$, for all $p>0$ \cite[Proposition 2.4]{DJN}.
 
\end{remark}

\subsection{Vertices of the Newton polyhedron} \label{locali}
As mentioned earlier, the vertices of the Newton polyhedron play a fundamental role in the description of the charts of the Nash blowup. This subsection is devoted to their study.\\

The vertices of $N_{\sigma}(\Gamma_p)$ can be classified into two types.

\begin{itemize}
    \item Interior vertices, which are of the form 
    \(v_{i-1}+v_i\) with \(i \in \{2,\dots , r-1\}\).
    \item Extremal vertices, corresponding to 
    \(v_0+v_1\) and \(v_{r-1}+v_r\).
\end{itemize}

In the case of interior vertices, the localization is given by  
\begin{eqnarray}\label{locint}
\check{\sigma}(v_{i-1},v_i) \;=\; 
\mathbb{R}_{\geq 0}\{\,v_{i-2}-v_i,\; v_{i+1}-v_{i-1}\,\}.
\end{eqnarray}
For extremal vertices, the corresponding cones are  
\begin{eqnarray}\label{locext}
\check{\sigma}(v_0,v_1)=\mathbb{R}_{\geq 0}\{\,v_0,\; v_2-v_0\,\}, 
\qquad 
\check{\sigma}(v_{r-1},v_r)=\mathbb{R}_{\geq 0}\{\,v_{r-2}-v_r,\; v_r\,\}.
\end{eqnarray}

In the following proposition, we describe the vertices of the Newton polyhedron in terms of continued fractions.

\begin{proposition}\label{verti}
Consider the continued fraction $[1,a_2,\dots, a_r]$ associated with a normal toric surface. Then, for each $2\leq i \leq r-1$, $v_{i-1}+v_i$ is not a vertex of $N_{\sigma}(\Gamma_p)$ if and only if $a_i=2$ and $a_{i+1}=2$. 
\end{proposition}
\begin{proof}
The point $v_{i-1}+v_i$ is not a vertex of $N_{\sigma}(\Gamma_p)$ if and only if the vectors $v_{i-2}-v_i$ and $v_{i+1}-v_{i-1}$ are linearly dependent.

By definition, we have
$$
p_i = a_i p_{i-1} - p_{i-2}, \qquad q_i = a_i q_{i-1} - q_{i-2}.
$$ 
Substituting these expressions into the coordinates of 
$v_{i-2}-v_i = (p_{i-2}-p_i, \, q_{i-2}-q_i)$, we obtain:
\begin{equation}\label{generator1}
    v_{i-2} - v_i = (2 p_{i-2} - a_i p_{i-1}, \, 2 q_{i-2} - a_i q_{i-1}).
\end{equation}

Similarly, using
$$
p_{i+1} = a_{i+1} p_i - p_{i-1}, \qquad q_{i+1} = a_{i+1} q_i - q_{i-1},
$$
and substituting into $v_{i+1}-v_{i-1} = (p_{i+1}-p_{i-1}, \, q_{i+1}-q_{i-1})$, we obtain:
\begin{equation}\label{generator2}
    v_{i+1} - v_{i-1} = (a_{i+1} p_i - 2 p_{i-1}, \, a_{i+1} q_i - 2 q_{i-1}).
\end{equation}

Thus, the vectors $v_{i-2}-v_i$ and $v_{i+1}-v_{i-1}$ are linearly dependent if and only if

$$
\det\begin{pmatrix}
    2p_{i-2} - a_i p_{i-1} & 2q_{i-2} - a_i q_{i-1} \\
    a_{i+1} p_i - 2p_{i-1} & a_{i+1} q_i - 2q_{i-1}
\end{pmatrix} = 0.
$$

Expanding the determinant, we obtain
\begin{equation}\label{ec4}
2 a_{i+1} (p_{i-2} q_i - p_iq_{i-2}) - a_i a_{i+1} (p_{i-1} q_i - p_i q_{i-1}) - 4 (p_{i-2} q_{i-1} - p_iq_{i-2}) = 0.
\end{equation}

By Proposition \ref{fc} (3), the term \(p_{i-2} q_i - p_iq_{i-2} \) corresponds to the denominator in lowest terms of the continued fraction \([a_{i-1},a_i]\), which is equal to \(a_i\). Using Proposition \ref{fc} (2) and substituting into (\ref{ec4}), we conclude that \(v_{i-1}+v_i\) is not a vertex of \(N_{\sigma}(\Gamma_p)\) if and only if \(a_i a_{i+1} = 4\). Since, by Remark \ref{obs1}, \(a_j > 1\) for all \(j > 1\), this is equivalent to saying that \(a_i = a_{i+1} = 2\).
\end{proof}

\begin{remark}
Alternatively, these vertices can also be described in terms of the vertices of the convex hull of the minimal generating set of the corresponding cone, as shown in \cite[Proposition~1.3]{DuarteSnoussi2025}.
\end{remark}


\section{One-step resolution}

In $\mathbb{R}^2$, every strongly convex rational cone of dimension two is generated by two integral vectors. These are called primitive if their coordinates are coprime. In the localization cones, the primitive vectors can be determined directly from the associated continued fraction.\\ 

For the proof of the following result, it is necessary to take Proposition \ref{fc} (2) into account, as it will be used repeatedly without further notice.

\begin{proposition}\label{primitivos}
Let $[1,a_2,\dots,a_r]$ be the continued fraction associated with the normal toric surface $X_{\sigma}$. 
Suppose that $v_{i-1}+v_i$ is a vertex of the Newton polyhedron $N_{\sigma}(\Gamma_p)$. 
Then the primitive generators of the localization cone are given by:  
\begin{enumerate}[label=(\arabic*)]
    \item $(p_{i-2}-m p_{i-1}, \, q_{i-2}-m q_{i-1})$ and $(n p_i - p_{i-1}, \, n q_i - q_{i-1})$, if $a_i=2m$ and $a_{i+1}=2n$.
    \item $(p_{i-2}-m p_{i-1}, \, q_{i-2}-m q_{i-1})$ and $(a_{i+1} p_i - 2p_{i-1}, \, a_{i+1} q_i - 2q_{i-1})$, if $a_i=2m$ and $a_{i+1}=2n+1$.
    \item $(2 p_{i-2} - a_i p_{i-1}, \, 2 q_{i-2} - a_i q_{i-1})$ and $(n p_i - p_{i-1}, \, n q_i - q_{i-1})$, if $a_i=2m+1$ and $a_{i+1}=2n$.
    \item $(2 p_{i-2} - a_i p_{i-1}, \, 2 q_{i-2} - a_i q_{i-1})$ and $(a_{i+1} p_i - 2p_{i-1}, \, a_{i+1} q_i - 2q_{i-1})$, if $a_i=2m+1$ and $a_{i+1}=2n+1$.
\end{enumerate}
\end{proposition}

\begin{proof}
Recall from the proof of Proposition~\ref{verti} the description of the generators of the corresponding localization cone. Note that the vectors \eqref{generator1} and \eqref{generator2} are not necessarily primitive, since this depends on the parity of $a_i$ and $a_{i+1}$. We now examine the different cases that may occur.\\

\medskip
\noindent\textbf{Case 1.} Suppose $a_i = 2m$ and $a_{i+1} = 2n$.  

In this case, $v_{i-2}-v_i$ and $v_{i+1}-v_{i-1}$ are not primitive, since $2$ divides all of their coordinates. Consider instead the vectors
$$
(p_{i-2}-m p_{i-1}, \, q_{i-2}-m q_{i-1}) \quad\text{and}\quad 
(n p_i - p_{i-1}, \, n q_i - q_{i-1}),
$$ 
which generate the same cone. These vectors are primitive, since
\begin{align*}
q_{i-1}(p_{i-2} - m p_{i-1}) - p_{i-1}(q_{i-2} - m q_{i-1}) 
&= p_{i-2} q_{i-1} - p_{i-1} q_{i-2} \\
&= 1, \\
-q_i(n p_i - p_{i-1}) + p_i(n q_i - q_{i-1}) 
&= p_{i-1} q_i - p_i q_{i-1} \\
&= 1.
\end{align*}

\medskip
\noindent\textbf{Case 2.} Suppose $a_i = 2m$ and $a_{i+1} = 2n+1$.  

Here $v_{i-2}-v_i$ is not primitive, so we take the primitive vector
$$(p_{i-2}-m p_{i-1}, \, q_{i-2}-m q_{i-1}).$$
In contrast, 
$$
v_{i+1}-v_{i-1} = ((2n+1)p_i - 2p_{i-1}, \, (2n+1) q_i - 2q_{i-1})
$$ 
is primitive, because
\begin{align*}
&(n q_i - q_{i-1}) \bigl[ (2n+1) p_i - 2 p_{i-1} \bigr] 
- (n p_i - p_{i-1}) \bigl[ (2n+1) q_i - 2 q_{i-1} \bigr] \\
&= (2n+1)(p_{i-1} q_i - p_i q_{i-1}) - 2n (p_{i-1} q_i - p_i q_{i-1}) \\
&= 1.
\end{align*}

\medskip
\noindent\textbf{Case 3.} Suppose $a_i = 2m+1$ and $a_{i+1} = 2n$.  

As in the previous case, we already know the primitive generator corresponding to $v_{i+1}-v_{i-1}$. Now we verify that
$$
(2 p_{i-2} - a_i p_{i-1}, \, 2 q_{i-2} - a_i q_{i-1})
$$
is primitive:
\begin{align*}
&(q_{i-2}-m q_{i-1}) \bigl[ 2 p_{i-2} - (2m+1) p_{i-1} \bigr] 
- (p_{i-2}-m p_{i-1}) \bigl[ 2 q_{i-2} - (2m+1) q_{i-1} \bigr] \\
&= (2m+1)(p_{i-2}q_{i-1}-p_{i-1}q_{i-2}) - 2m(p_{i-2}q_{i-1}-p_{i-1}q_{i-2}) \\
&= 1.
\end{align*}

\medskip
\noindent\textbf{Case 4.} Suppose $a_i = 2m+1$ and $a_{i+1} = 2n+1$.  

The argument is entirely analogous to the previous cases, and one verifies that the vectors
$$
(2 p_{i-2} - a_i p_{i-1}, \, 2 q_{i-2} - a_i q_{i-1}) \quad\text{and}\quad 
(a_{i+1} p_i - 2p_{i-1}, \, a_{i+1} q_i - 2q_{i-1})
$$
are primitive.
\end{proof}

\subsection{One-step resolution via normalized Nash blowup}

\begin{theorem}\label{nashnormal}
Let $\operatorname{char}(\mathbb{K})=p\geq 0$.  
The normalized Nash blowup of a normal toric surface is smooth if and only if its associated continued fraction 
$
[1,a_2,\dots ,a_r]
$
is one of the following: 

\begin{enumerate}[label=(\arabic*)]
    \item $[1,2]$ if $r=2$.
    \item $[1,2,2]$ if $r=3$.
    \item $[1,2,2,2]$, $[1,2,3,2]$, or $[1,2,4,2]$ if $r=4$.
    \item $[1,2,a_3,\dots, a_{r-1},2]$ for $r \geq 5$, where $(a_i,a_{i+1}) \in \{(2,2),(2,3),(2,4),(3,2),(4,2)\}$ for all $3 \leq i \leq r-2$.
\end{enumerate}
\end{theorem}

The proof of this theorem will be divided into Propositions \ref{res1}, \ref{res2}, and \ref{res3}. 

\begin{proposition} \label{res1}
Let $[1,a_2,a_3,\dots, a_r]$ be the continued fraction associated to a normal toric surface. Then the localization of $N_{\sigma}(\Gamma_p)$ at the vertex $v_0+v_1$ is smooth if and only if $a_2=2$. 
\end{proposition}

\begin{proof}
According to the definition in (\ref{suce}), we have
$$
\begin{aligned}
    p_0 &= 1,   &\quad p_1 &= 1,     &\quad p_2 &= a_2 - 1, \\
    q_0 &= 0,   &\quad q_1 &= 1,     &\quad q_2 &= a_2.
\end{aligned}
$$
Consequently,
$$
v_0+v_1=(2,1), 
\qquad 
v_1+v_2=(a_2,a_2+1).
$$
Therefore, the localization at the vertex $v_0+v_1$ is the cone 
$$
\mathbb{R}_{\geq 0}\{(1,0),(a_2-2,a_2)\}.
$$

As in the proof of Proposition \ref{primitivos}, the primitiveness of the vector $(a_2-2,a_2)$ depends on the parity of $a_2$. 

\medskip
\noindent\textbf{Case 1.} If $a_2=2n$, then $(a_2-2,a_2)=(2n-2,2n)$ is not primitive, and we replace it by $(n-1,n)$. The resulting cone is smooth if and only if 
$$
\det \begin{pmatrix}
    1 & 0\\
    n-1 & n
\end{pmatrix}=1,
$$
which occurs if and only if $n=1$, that is, $a_2=2$. 

\medskip
\noindent\textbf{Case 2.} If $a_2=2n+1$, then $(a_2-2,a_2)=(2n-1,2n+1)$ is primitive. In this case, the cone is smooth if and only if 
$$
\det \begin{pmatrix}
    1 & 0\\
    2n-1 & 2n+1
\end{pmatrix}=1,
$$
which is equivalent to $a_2=1$. However, this possibility is excluded since $a_2>1$.  

\medskip
Thus, the only possibility for the cone to be smooth is $a_2=2$. 
\end{proof}

\begin{proposition}\label{res2}
Let $[1, a_2, \dots, a_r]$ be a continued fraction associated with a normal toric surface. The localization of $N_{\sigma}(\Gamma_p)$ at the vertex $v_{r-1} + v_r$ is smooth if and only if $a_r = 2$.
\end{proposition}
\begin{proof}
As in (\ref{locext}), the localization of the Newton polyhedron at the vertex $v_{r-1}+v_r$ is given by the cone
$$
\check{\sigma}(v_{r-1}+v_r)
=\mathbb{R}_{\geq 0}\{\,v_{r-2}-v_r,\,(p_r,q_r)\,\}.
$$
Observe that $(p_r,q_r)=(P,Q)$, so $(p_r,q_r)$ is a primitive vector. Using (\ref{generator1}), we have:
$$
\check{\sigma}(v_{r-1}+v_r)
=\mathbb{R}_{\geq 0}\{(2p_{r-2}-a_rp_{r-1},\,2q_{r-2}-a_rq_{r-1}),\,(p_r,q_r)\}.
$$
Hence, the primitiveness of the first generator depends on the parity of $a_r$, which leads us to consider two cases.

\medskip
\noindent\textbf{Case 1.} Suppose that $a_r=2n$.  
In this case, the first generator of the cone is not primitive. As in Proposition \ref{primitivos}, we replace it by the vector
$$
(p_{r-2}-np_{r-1},\,q_{r-2}-nq_{r-1}),
$$
which is primitive. The cone is smooth if and only if
$$
\det \begin{pmatrix}
    p_{r-2}-np_{r-1} & q_{r-2}-nq_{r-1}\\
    p_r & q_r
\end{pmatrix}=1.
$$
Expanding the determinant, this condition becomes
$$
(p_{r-2}q_r-p_rq_{r-2})-n(p_{r-1}q_r-p_rq_{r-1})=1.
$$
By Proposition \ref{fc} (2) and (3), we conclude that the cone is smooth if and only if $a_r-n=1$. Therefore, the cone is smooth precisely when $a_r=2$.

\medskip
\noindent\textbf{Case 2.} Suppose that $a_r=2n+1$.  
Here, the first generator is already primitive. In this case, the cone is smooth if and only if
$$
2(p_{r-2}q_r-p_rq_{r-2})-a_r(p_{r-1}q_r-p_rq_{r-1})=1,
$$
which would force $a_r=1$. However, this is impossible due to Remark \ref{obs1}. We therefore conclude that this case does not occur.
\end{proof}

\begin{proposition}\label{res3}
Let $[1,a_2,\dots,a_{i-1},a_i,a_{i+1},\dots,a_r]$ be the continued fraction associated to a normal toric surface, such that $v_{i-1}+v_i$ is a vertex of $N_{\sigma}(\Gamma_p)$ for some $2\leq i\leq r-1$ and $r \geq 4$. Then the localization at $v_{i-1}+v_i$ is smooth if and only if one of the following conditions holds:
    \begin{enumerate}[label=(\arabic*)]
        \item $a_i=2$ and $a_{i+1}=3$.
        \item $a_i=2$ and $a_{i+1}=4$.
        \item $a_i=3$ and $a_{i+1}=2$.
        \item $a_i=4$ and $a_{i+1}=2$.
    \end{enumerate}
\end{proposition}
\begin{proof}
As in (\ref{locint}), the localization at the vertex $v_{i-1}+v_i$ is given by
$$
\check{\sigma}(v_{i-1} + v_i) 
= \mathbb{R}_{\geq 0}\{\,v_{i-2} - v_i,\; v_{i+1} - v_{i-1}\,\},
$$ 
where
\begin{align*}
    v_{i-2}-v_i &= (2p_{i-2} - a_i p_{i-1}, \, 2q_{i-2} - a_i q_{i-1}), \\
    v_{i+1}-v_{i-1} &= (a_{i+1} p_i - 2p_{i-1}, \, a_{i+1} q_i - 2q_{i-1}).
\end{align*}
As before, the analysis depends on the parity of $a_i$ and $a_{i+1}$. We consider the following cases.

\medskip
\noindent\textbf{Case 1.} $a_i=2m$ and $a_{i+1}=2n$.  
According to Proposition \ref{primitivos} (1), the primitive generators are
$$
(p_{i-2}-mp_{i-1},\,q_{i-2}-mq_{i-1})
\quad\text{and}\quad
(np_i-p_{i-1},\,nq_i-q_{i-1}).
$$
Hence, the cone is smooth if and only if
$$
\det\begin{pmatrix}
p_{i-2} - m p_{i-1} & q_{i-2} - m q_{i-1} \\
n p_i - p_{i-1} & n q_i - q_{i-1}
\end{pmatrix} = 1,
$$
which reduces to
$$
n(p_{i-2}q_i-p_iq_{i-2})-nm(p_{i-1}q_i-p_iq_{i-1})-(p_{i-2}q_{i-1}-p_{i-1}q_{i-2})=1.
$$
By Proposition \ref{fc} (2) and (3), this is equivalent to $2=na_i-nm$. We conclude that the cone is smooth if and only if $(a_i,a_{i+1})=(2,4)$ or $(4,2)$.

\medskip
\noindent\textbf{Case 2.} $a_i=2m$ and $a_{i+1}=2n+1$.  
By Proposition \ref{primitivos} (2), the primitive generators are
$$
(p_{i-2}-mp_{i-1},\,q_{i-2}-mq_{i-1})
\quad\text{and}\quad
(a_{i+1}p_i-2p_{i-1},\,a_{i+1}q_i-2q_{i-1}).
$$
The cone is smooth if and only if
$$
\det\begin{pmatrix}
    p_{i-2} - m p_{i-1} & q_{i-2} - m q_{i-1} \\
a_{i+1} p_i - 2p_{i-1} & a_{i+1} q_i - 2q_{i-1}
\end{pmatrix} = 1,
$$
which becomes
$$
a_{i+1}(p_{i-2}q_i-p_iq_{i-2})-a_{i+1}m(p_{i-1}q_i-p_iq_{i-1})-2(p_{i-2}q_{i-1}-p_{i-1}q_{i-2})=1.
$$
Using Proposition~\ref{fc} again, this condition
is equivalent to $a_i a_{i+1}-a_{i+1}m=3$, 
which occurs exactly when $(a_i,a_{i+1})=(2,3)$.

\medskip
\noindent\textbf{Case 3.} $a_i=2m+1$ and $a_{i+1}=2n$.  
In this case, the cone is smooth if and only if
$$
\det\begin{pmatrix}
     2p_{i-2} - a_i p_{i-1} & 2q_{i-2} - a_i q_{i-1} \\
n p_i - p_{i-1} & n q_i - q_{i-1}
\end{pmatrix} =1,
$$
that is,
$$
2n(p_{i-2}q_i-p_iq_{i-2})-a_in(p_{i-1}q_i-p_iq_{i-1})-2(p_{i-2}q_{i-1}-p_{i-1}q_{i-2})=1.
$$
Applying Proposition \ref{fc} once more, we get the condition $a_in=3$, so the cone is smooth exactly when $(a_i,a_{i+1})=(3,2)$.

\medskip
\noindent\textbf{Case 4.} $a_i=2m+1$ and $a_{i+1}=2n+1$.  
A similar computation shows that the smoothness condition would require $a_ia_{i+1}=5$, which would force $(a_i,a_{i+1})=(1,5)$ or $(5,1)$. However, $a_i>1$ for all $i>1$, so this case is impossible.

\medskip

In conclusion, the localization at $v_{i-1}+v_i$ is smooth if and only if one of the cases listed in the proposition holds.
\end{proof}

\begin{proof}[Proof of Theorem \ref{nashnormal}]
Let $X_{\sigma}$ be a normal toric surface, and let $[1,a_2,\dots ,a_r]$ be its associated continued fraction.  
We shall determine the possible forms of this fraction according to its length $r$.

\begin{itemize}
   \item \textbf{Case $r=2$.}  
In this case, the Newton polyhedron has only two vertices, both of which are extremal: $v_0 + v_1$ and $v_1 + v_2$.  
By Propositions~\ref{res1} and~\ref{res2}, both localizations are smooth if and only if the continued fraction is of the form $[1,2]$.

    \item \textbf{Case $r=3$.}  
    Again, by Propositions \ref{res1} and \ref{res2}, the normalized Nash blowup of $X_{\sigma}$ is smooth if and only if the continued fraction is $[1,2,2]$.

\item \textbf{Case $r=4$.}  
In this case there are exactly four possible vertices of the Newton polyhedron:
$$
v_0+v_1,\qquad v_1+v_2,\qquad v_2+v_3,\qquad v_3+v_4.
$$

To ensure smoothness in the first and last localization, corresponding to $v_0+v_1$ and $v_3+v_4$, Propositions~\ref{res1} and \ref{res2} imply, and are in fact equivalent to requiring, that $a_2 = 2$ and $a_4 = 2$.

Therefore, smoothness in the remaining two localizations depends only on the value of $a_3$.
Proposition~\ref{verti} states that $v_2+v_3$ is not a vertex of $N_\sigma(\Gamma_p)$ if and only if $a_3 = 2$ and $a_4 = 2$.

Since we already know that $a_4 = 2$, it follows that $a_3 = 2$ if and only if the point $v_2+v_3$ is not a vertex of the Newton polyhedron. In this case, the only two relevant localizations are smooth, and therefore the normalized Nash blowup is smooth. This corresponds to the continued fraction $[1,2,2,2]$.\\
On the other hand, if $v_2+v_3$ \emph{is} a vertex, then Proposition~\ref{res3}, together with $a_4 = 2$, implies that the only possible values for $a_3$ are $3$ or $4$. In both cases, the previous propositions ensure that the corresponding localizations are smooth, and again the normalized Nash blowup is smooth.\\

In conclusion, for a continued fraction of length $4$, the normalized Nash blowup is smooth if and only if the associated continued fraction is
$$
[1,2,2,2],\qquad [1,2,3,2],\qquad \text{or} \qquad [1,2,4,2].
$$

    \item \textbf{Case $r \geq 5$.} 
As in the previous cases, the localizations corresponding to the extremal vertices are smooth if and only if $a_2 = 2$ and $a_r = 2$. For the intermediate indices, it is enough to determine whether the point $v_{i-1}+v_i$ is a vertex of the Newton polyhedron. As before, this point is not a vertex if and only if $(a_i,a_{i+1})=(2,2)$; otherwise, the localization at $v_{i-1}+v_i$ is smooth if and only if
$$
(a_i,a_{i+1}) \in \{(2,3),(2,4),(3,2),(4,2)\}.
$$

Therefore, for continued fractions of length $r \ge 5$, the normalized Nash blowup is smooth if and only if the associated continued fraction is of the form
$$
[1,2,a_3,a_4,\dots,a_{r-1},2],
\qquad \text{with } (a_i,a_{i+1}) \in \{(2,2),(2,3),(2,4),(3,2),(4,2)\}.
$$
\end{itemize}
\end{proof}


\subsection{One-step resolution via Nash blowup}

In this subsection, we analyze the non-normalized Nash blowup. In this case, 
we do not obtain the same result as in Theorem \ref{nashnormal};
a key difference appears in the cases $(a_{i-1},a_i)\in\{(2,4),(4,2)\}$ (see Example \ref{ejemplo}). Here we determine the conditions under which an analogous result holds.
Unlike the previous subsection, the results presented here are formulated in characteristic zero.

In Section \ref{descri}, 
we described the semigroups associated with 
the affine charts of the Nash blowup. 
Given a vertex $v_{i-1}+v_i$ of $N_{\sigma}(\Gamma_0)$, we denote by $\Gamma(v_{i-1},v_i)$ 
the semigroup generated by
$$
\{\, v_{i-1}, v_i \,\} \cup A(v_{i-1}) \cup A(v_i), 
$$
where
\begin{align*}
    A(v_{i-1}) &= \{\, v_j - v_{i-1} \;\mid\; 
    j\in \{0,\dots, r\}\setminus \{\,i-1,i\,\},\; \det(v_j \, v_i)\neq 0 \,\},\\ 
    A(v_{i}) &= \{\, v_k - v_i \;\mid\; 
    k\in \{0,\dots, r\}\setminus \{\,i-1,i\,\},\; \det(v_k \, v_{i-1})\neq 0 \,\}. 
\end{align*}

Note that $\det(v_j\, v_i)=p_j q_i - p_i q_j$.  
By Proposition \ref{fc} (3), this determinant is nonzero.
Hence, in characteristic zero, the previous sets simplify to
\begin{align*}
    A(v_{i-1}) &= \{\, v_j - v_{i-1} \;\mid\; 
    j\in \{0,\dots, r\}\setminus \{\,i-1,i\,\} \,\},\\ 
    A(v_{i}) &= \{\, v_k - v_i \;\mid\; 
    k\in \{0,\dots, r\}\setminus \{\,i-1,i\,\} \,\}.
\end{align*}

\begin{example}\label{ejemplo}
Consider the continued fraction $[1,2,4,2]$. In this case we have
$$
v_0 = (1,0),\quad 
v_1 = (1,1),\quad 
v_2 = (1,2),\quad 
v_3 = (3,7),\quad 
v_4 = (5,12). 
$$
Observe that $v_1 + v_2$ is a vertex of the Newton polyhedron; thus,
$$
A(v_1)=\{(0,-1), (2,6), (4,11)\},\qquad 
A(v_2)=\{(0,-2), (2,5), (4,11)\}.
$$

After removing redundancies,
the semigroup $\Gamma(v_1,v_2)$ is generated by 
$$
\{(0,-1), (1,2), (2,6)\}. 
$$

\noindent On the other hand, the localization of the Newton polyhedron $\check{\sigma}(v_1,v_2)$ 
has as primitive generators the vectors $(0,-1)$ and $(1,3)$, 
since $(1,3)$ is the primitive vector in the direction of $(2,6)$. 

\noindent Notice that the primitive generator $(1,3)$ does not belong to $\Gamma(v_1,v_2)$, 
because it cannot be written as a nonnegative integer combination of the generators above. 
Therefore, this semigroup is not saturated, and consequently the toric surface 
$X_{\Gamma(v_1,v_2)}$ cannot be smooth. Hence, the non-normalized Nash blowup of the toric 
variety associated with the continued fraction $[1,2,4,2]$ is not smooth. 
\end{example}

The previous example shows that the behavior observed in the normalized case does not extend 
to the non-normalized Nash blowup. We now state the second main result of this work, 
which describes precisely which continued fractions give rise to a smooth Nash blowup. 

\begin{theorem}\label{nash}
 Assume $\operatorname{char}(\mathbb{K})=0$.  
The Nash blowup of a normal toric surface is smooth if and only if its associated continued fraction 
\(
[1,a_2,\dots,a_r]
\) 
is one of the following:

\begin{enumerate}[label=(\arabic*)]
    \item $[1,2]$ if $r=2$.
    \item $[1,2,2]$ if $r=3$.
    \item $[1,2,2,2]$ or  $[1,2,3,2]$, if $r=4$.
    \item $[1,2,a_3,\dots, a_{r-1},2]$ for $r \geq 5$, where $(a_i,a_{i+1}) \in \{(2,2),(2,3),(3,2)\}$ for all $3 \leq i \leq r-2$.
\end{enumerate}
\end{theorem}

As before, the proof of this result will be divided into several steps.

\begin{proposition}\label{res4}
Let $[1,a_2,\dots , a_r]$ be the continued fraction associated to the normal toric surface $X_{\sigma}$. 
Suppose that $v_{i-1}+v_i$ is a vertex of the Newton polyhedron $N_{\sigma}(\Gamma_0)$ for $2<i<r-1$. 
If $(a_i,a_{i+1})\in \{(2,3),(3,2)\}$, then the affine chart of the Nash blowup $X_{\Gamma(v_{i-1},v_i)}$ is smooth.
\end{proposition}

\begin{proof}
    By Theorem \ref{nashnormal}, if $(a_i,a_{i+1})\in \{(2,3),(3,2)\}$, then 
    $X_{\sigma(v_{i-1},v_i)}$ is smooth. 
    Therefore, it suffices to show that 
    $$
    \Gamma(v_{i-1},v_i) = \check{\sigma}(v_{i-1},v_i)\cap \mathbb{Z}^2.
    $$
And to show that these semigroups coincide, it is enough to check that the primitive generators of $\check{\sigma}(v_{i-1},v_i)$ lie in $\Gamma(v_{i-1},v_i)$, since their determinant is equal to one.
    
    \medskip
    \noindent\textbf{Case 1.} 
    If $(a_i,a_{i+1})=(2,3)$, Proposition \ref{primitivos} (2), states that the primitive generators of 
    $\check{\sigma}(v_{i-1},v_i)$ are $$(p_{i-2}- p_{i-1}, \, q_{i-2}- q_{i-1}) \quad\text{and}\quad (3 p_i - 2p_{i-1}, \, 3 q_i - 2q_{i-1}),$$ 
    which are equal to $v_{i-2}-v_{i-1}$ and $v_{i+1}-v_{i-1}$, respectively.
    Note that both vectors belong to $A(v_{i-1})$, and hence to $\Gamma(v_{i-1},v_i)$.
    
    \medskip
    \noindent\textbf{Case 2.} 
  If $(a_i, a_{i+1}) = (3,2)$, Proposition \ref{primitivos} (3) ensures that the primitive generators of 
$\check{\sigma}(v_{i-1},v_i)$ are 
$$
(2p_{i-2} - 3p_{i-1},\, 2q_{i-2} - 3q_{i-1}) 
\quad \text{and} \quad
(p_i - p_{i-1},\, q_i - q_{i-1}),
$$
which are equal to $v_{i-2} - v_i$ and $v_i - v_{i-1}$, respectively.

\noindent The first vector lies in $A(v_i)$ and the second in $A(v_{i-1})$, and hence both belong to $\Gamma(v_{i-1},v_i)$.\\
    
    In both cases we conclude that $X_{\Gamma(v_{i-1},v_i)}$ is smooth.
\end{proof}

\begin{proposition}\label{res5}
    Let $[1,a_2,\dots, a_r]$ be the continued fraction associated to the normal toric surface $X_{\sigma}$. Then:
    \begin{enumerate}[label=(\arabic*)]
        \item $X_{\Gamma(v_0,v_1)}$ is smooth if and only if $a_2=2$.
        \item $X_{\Gamma(v_{r-1},v_r)}$ is smooth if and only if $a_r=2$.
    \end{enumerate}
\end{proposition}

\begin{proof}
    \noindent \textbf{(1)} First, suppose that $X_{\Gamma(v_0,v_1)}$ is smooth. In particular, it is normal, and therefore 
    $$
    X_{\Gamma(v_0,v_1)} = X_{\sigma(v_0,v_1)}.
    $$
    By Proposition \ref{res1}, this implies that $a_2=2$.  

    Conversely, if $a_2=2$, then by Proposition \ref{res1}, $X_{\sigma(v_0,v_1)}$ is smooth. To conclude that $X_{\Gamma(v_0,v_1)}$ is also smooth, it suffices to verify that
    $$
    \Gamma(v_0,v_1) = \check{\sigma}(v_0,v_1)\cap \mathbb{Z}^2.
    $$
    In the proof of Proposition \ref{res1}, we saw that the primitive generators of $\check{\sigma}(v_0,v_1)$ are 
    $$
    v_0=(1,0) \quad \text{and} \quad (0,1).
    $$
    Since $v_0\in \Gamma(v_0,v_1)$, it only remains to check that $(0,1)\in \Gamma(v_0,v_1)$. Note that
    $$
    v_2-v_1 = (a_2-1,a_2)-(1,1) = (a_2-2, a_2-1) = (0,1).
    $$
    Since $(0,1)\in A(v_1)$, we conclude that $(0,1)\in \Gamma(v_0,v_1)$.  

    \medskip
    \noindent \textbf{(2)} The proof of this case is analogous to the previous one.
\end{proof}

\begin{proposition}\label{res6}
   Let $[1,a_2,\dots , a_r]$ be the continued fraction associated to the normal toric surface $X_{\sigma}$. 
Suppose that $v_{i-1}+v_i$ is a vertex of the Newton polyhedron $N_{\sigma}(\Gamma_0)$ for $2<i<r-1$. 
If $(a_i,a_{i+1})\in \{(2,4),(4,2)\}$, then the affine chart of the Nash blowup $X_{\Gamma(v_{i-1},v_i)}$ is not smooth.
\end{proposition}

\begin{proof}
It suffices to show that $\Gamma(v_{i-1},v_i)$ is not saturated.  
We will consider the case $(a_i,a_{i+1})=(2,4)$, the case $(4,2)$ is analogous.

By Proposition~\ref{primitivos} (1), the primitive generators of the cone $\check{\sigma}(v_{i-1},v_i)$ are
$$
u_1 = v_{i-2}-v_{i-1}
\qquad\text{and}\qquad
u_2 = 2v_i - v_{i-1}.
$$
We claim that $u_2 \notin \Gamma(v_{i-1},v_i)$. To show this, suppose the contrary.  
Consider the ray $\tau = \mathbb{R}_{\ge 0}\{u_2\}$. Then $\tau \cap \Gamma(v_{i-1},v_i)$ is a face of $\Gamma(v_{i-1},v_i)$ (see \cite[Lemma~12]{GT}). Since $u_2$ is a primitive generator of $\tau$, it is irreducible in the face. Hence, if $u_2 \in \Gamma(v_{i-1},v_i)$, it must be one of the generators of $\Gamma(v_{i-1},v_i)$.

The generators of $\Gamma(v_{i-1},v_i)$ are 
$$
\{v_{i-1},v_i\} \ \cup\ A(v_{i-1}) \ \cup\ A(v_i).
$$  
We now check that $u_2$ cannot coincide with any of these generators:

\begin{itemize}
  \item Clearly, $u_2 \neq v_i$ and $u_2 \neq v_{i-1}$; otherwise we would have $v_{i-1} = v_i$, which is impossible by definition.
  
  \item If $u_2 \in A(v_{i-1})$, then there would exist some $j \notin \{i-1,i\}$ such that
  $v_j - v_{i-1} = u_2$, so $v_j = 2v_i - v_{i-1} + v_{i-1} = 2v_i$.  
  But this cannot occur, since these elements are irreducible.
  
  \item If $u_2 \in A(v_i)$, then there would exist some $j \notin \{i-1,i\}$ such that  
  $v_j - v_i = u_2$, and thus $v_j = 3v_i - v_{i-1}$.  
  Since $v_{i+1} = 4v_i - v_{i-1}$, this would imply 
  \[
  v_{i+1} = v_j + v_i,
  \]
  contradicting the irreducibility of $v_{i+1}$.
\end{itemize}

Thus, $u_2$ cannot be any of the listed generators, and therefore $u_2 \notin \Gamma(v_{i-1},v_i)$. However,
$$
2u_2 = 4v_i - 2v_{i-1} = v_{i+1} - v_{i-1},
$$ 
which does lie in the semigroup. Hence, $\Gamma(v_{i-1},v_i)$ is not saturated.  
We conclude that the affine chart $X_{\Gamma(v_{i-1},v_i)}$ is not normal, and therefore cannot be smooth.
\end{proof}

\begin{proof}[Proof of Theorem \ref{nash}]
Let $X_{\sigma}$ be a toric surface whose Nash blowup is smooth, and let $[1,a_2,\dots, a_r]$ be the associated continued fraction. 

For each vertex $v_{i-1}+v_i$ of $N_\sigma(\Gamma_0)$, the localization $X_{\Gamma(v_{i-1},v_i)}$ is smooth and hence normal. Thus,
$$
X_{\Gamma(v_{i-1},v_i)} = X_{\sigma(v_{i-1},v_i)}.
$$

It follows that the normalized Nash blowup of $X_{\sigma}$ is also smooth. Applying Theorem \ref{nashnormal} together with Proposition \ref{res6}, the continued fraction $[1,a_2,\dots,a_r]$ must have one of the forms listed in the statement.

The converse follows immediately from Propositions \ref{res4} and \ref{res5}.
\end{proof}


\bibliographystyle{plain} 
\bibliography{Reference}

\vspace{.5cm}
\noindent{\footnotesize \textsc {Amador Cruz-Fuentes, Centro de Ciencias Matem\'aticas, UNAM Campus Morelia.} \\ E-mail: acruz@matmor.unam.mx}\\

\end{document}